\documentclass[12pt, a4paper]{amsart}
\usepackage[utf8]{inputenc}
\usepackage{amsxtra,amssymb,amsthm,amsmath,amscd,mathrsfs, epsfig, eufrak}
\usepackage{amscd, amsmath, mathrsfs, amssymb, amsthm, amsxtra, bbding, epsfig, eucal, eufrak, graphicx, latexsym, mathrsfs, mathbbol, bbold}
\usepackage[all]{xy}

\usepackage{tikz}

\usepackage[normalem]{ulem}
\usepackage{soul}
\usepackage{color}

\setstcolor{red}

\makeatletter
\@namedef{subjclassname@2010}{%
\textup{2010} Mathematics Subject Classification}
\makeatother

\def \N {{\mathbb N}}

\def \R {{\mathbb R}}

\def\le{\leqslant}
\def\ge{\geqslant}

\theoremstyle{plain}
\newtheorem{theorem}{Theorem}
\newtheorem{proposition}{Proposition}[section]
\newtheorem{lemma}[proposition]{Lemma}

\theoremstyle{remark}

\numberwithin{equation}{section}

\addtocounter{footnote}{1}

\begin{document}


\title[A variant of the prime number theorem]
{A variant of the prime number theorem}
\author[\tiny Kui Liu, Jie Wu \& Zhishan Yang]{Kui Liu, Jie Wu \& Zhishan Yang}

\address{%
Kui Liu
\\
School of Mathematics and Statistics
\\
Qingdao University
\\
308 Ningxia Road
\\
Qingdao
\\
Shandong 266071
\\
China}
\email{liukui@qdu.edu.cn}

\address{%
Jie Wu
\\
CNRS UMR 8050\\
Laboratoire d'Analyse et de Math\'ematiques Appliqu\'ees\\
Universit\'e Paris-Est Cr\'eteil\\
94010 Cr\'eteil cedex\\
France}
\email{jie.wu@u-pec.fr}

\address{%
Zhishan Yang
\\
School of Mathematics and Statistics
\\
Qingdao University
\\
308 Ningxia Road
\\
Qingdao
\\
Shandong 266071
\\
China}
\email{zsyang@qdu.edu.cn}

\date{\today}

\begin{abstract}
Let $\Lambda(n)$ be the von Mangoldt function, and let $[t]$ be the integral part of real number $t$.
In this note, we prove that for any $\varepsilon>0$ the asymptotic formula
$$
\sum_{n\le x} \Lambda\Big(\Big[\frac{x}{n}\Big]\Big)
= x\sum_{d\ge 1} \frac{\Lambda(d)}{d(d+1)} + O_{\varepsilon}\big(x^{9/19+\varepsilon}\big)
\qquad
(x\to\infty)$$
holds.
This improves a recent result of Bordell\`es, which requires $\frac{97}{203}$ 
in place of $\frac{9}{19}$.
\end{abstract}


\maketitle

\section{Introduction}

The prime number theorem is a basic result in number theory and has many applications.
Denoting by $\Lambda(n)$ the von Mangoldt function, 
the prime number theorem states, in strong form, as follows: 
there is a constant $c>0$ such that for $x\to\infty$, we have
$$
\sum_{n\le x} \Lambda(n) = x + O(x\exp\{-c(\log x)^{3/5}(\log\log x)^{-1/5}\})
$$
and
$$
\sum_{n\le x} \Lambda(n) = x + O_{\varepsilon}(x^{1/2+\varepsilon})
\;\;\Leftrightarrow\;\;
\text{Riemann Hypothesis},
$$ 
where $\varepsilon$ is an arbitrarily small positive constant.
Clearly it is also interesting to study the distribution of prime numbers in different sequences of integers such as the arithmetic progressions, 
the Beatty sequence $([\alpha n + \beta])_{n\ge 1}$, 
the Piatetski-Shapiro sequence $([n^c])_{n\ge 1}$, etc,
where $[t]$ denotes the integral part of the real number.
For example, Banks and Shparlinski \cite[Corollary 5.6]{BanksShparlinski2009} proved the following result:
Let $\alpha$ and $\beta$ be fixed real numbers with $\alpha>0$, irrational and of finite type,
then there is a positive constant $c=c(\alpha, \beta)$ such that
$$
\sum_{n\le x} \Lambda([\alpha n + \beta])
= x + O(x\exp\{-c(\log x)^{3/5}(\log\log x)^{-1/5}\})
$$
as $x\to\infty$.
About works related to the Piatetski-Shapiro prime number theorem,
we infer the reader to see \cite{PiatetskiShapiro1953, HeathBrown1983, RivatWu2001}.
On the other hand, Bordell\`es-Dai-Heyman-Pan-Shparlinski \cite{BDHPS2019} 
established an asymptotic formula of 
\begin{equation}\label{Sf}
S_f(x) := \sum_{n\le x} f\Big(\Big[\frac{x}{n}\Big]\Big),
\end{equation}
under some simple assumptions of $f$.
Subsequently, Wu \cite{Wu2020} and Zhai \cite{Zhai2020} improveed their results independently.
In particular, applying \cite[Theorem 1.2(i)]{Wu2020} or \cite [Theorem 1]{Zhai2020} 
to the von Mangoldt function $\Lambda(n)$, we have
$$
S_{\Lambda}(x) 
= x\sum_{d\ge 1}\frac{\Lambda(d)}{d(d+1)} + O_{\varepsilon}\big(x^{1/2+\varepsilon}\big)
$$
for $x\ge 1$.
With the help of the Vaughan identity and the method of one-dimensional exponential sum,
Ma and Wu \cite{MaWu2020} breaked the $\frac{1}{2}$-barrier:
$$
S_{\Lambda}(x)
= x\sum_{d\ge 1}\frac{\Lambda(d)}{d(d+1)} + O_{\varepsilon}\big(x^{35/71+\varepsilon}\big)
\qquad
(x\ge 1).
$$
Very recently Bordell\`es \cite[Corollary 1.3]{Bordelles2020} sharpened the exponent $\frac{35}{71}$ 
to $\frac{97}{203}$ by using a result of Baker on 2-dimensional exponential sums \cite[Theorem 6]{Baker2007}.

The aim of this short note is to propose a better exponent 
by establishing an estimate on 3-dimensional exponential sums (see Proposition \ref{prop:HeatBrown} below).

\begin{theorem}\label{thm}
For any $\varepsilon>0$, we have
\begin{equation}\label{eq:thm}
S_{\Lambda}(x)
= x\sum_{d\ge 1}\frac{\Lambda(d)}{d(d+1)} + O_{\varepsilon}\big(x^{9/19+\varepsilon}\big)
\quad
\text{as}\;\;
x\to \infty.
\end{equation}
\end{theorem}

For comparaison, we have $\frac{97}{203}\approx 04778$ and $\frac{9}{19}\approx 0.4736$.
In fact, Bordell\`es established a more general result (see \cite[Theorem 1.1]{Bordelles2020}):
\begin{equation}\label{Bordelles}
S_{\Lambda}(x)
= x\sum_{d\ge 1}\frac{\Lambda(d)}{d(d+1)} 
+ O_{\varepsilon}\big(x^{\frac{14(\kappa+1)}{29\kappa-\lambda+30}+\varepsilon}\big)
\quad
(x\to \infty),
\end{equation}
where $(\kappa, \lambda)$ is an exponent pair satisfying $\kappa\le \frac{1}{6}$ and $\lambda^2+\lambda+3-\kappa (5+9\kappa-\lambda)>0$. 
The exponent $\frac{97}{203}$ comes from the choice of 
$(\kappa, \lambda) = (\frac{13}{84}+\varepsilon, \frac{55}{84}+\varepsilon)$
--- Bourgain's new exponent pair \cite[Theorem 6]{Bourgain2017}.
We note that under the exponent pair hypothesis 
(i.e. $(\varepsilon, \frac{1}{2}+\varepsilon)$ is an exponent pair, see \cite{GrahamKolesnik1991}),
Bordell\`es' \eqref{Bordelles} only gives the exponent $\frac{28}{59}\approx 0.4745$, 
which is larger than our constant $\frac{9}{19}\approx 0.4736$.

Some related works on the quantity \eqref{Sf} can be found in \cite{MaSun2020, Wu2019}.

\vskip 8mm

\section{Preliminary lemmas}

In this section, we shall cite three lemmas, which will be needed in the next section.
The first one is \cite[Lemma 1]{FouvryIwaniec1989}.

\begin{lemma}\label{lem:space}
Let $\alpha\in \R^*$ and $\beta\in \R^*$.
For $M\ge 1$, $N\ge 1$ and $\Delta>0$, define
$$
\mathcal{D}(M, N; \Delta)
:= \Big|\Big\{(m, \tilde{m}, n, \tilde{n})\in \N^4 : m, \tilde{m}\sim M; \; n, \tilde{n}\sim N; \;
\Big|\Big(\frac{m}{\tilde{m}}\Big)^{\alpha} - \Big(\frac{n}{\tilde{n}}\Big)^{\beta}\Big|\le \Delta\Big\}\Big|,
$$
where $m\sim M$ means that $M<m\le 2M$.
Then we have 
$$
\mathcal{D}(M, N)\ll_{\alpha, \beta} MN \log(2MN) + \Delta (MN)^2
$$
uniformly for $M\ge 1$, $N\ge 1$ and $\Delta>0$,
where the implied constant depends on $\alpha$ and $\beta$.
\end{lemma}

The second one is the Vaughan identity \cite[formula (3)]{Vaughan1980}.

\begin{lemma}\label{lem:VaughanIdentity}
There are six real arithmetical functions $\alpha_k(n)$ verifying
$|\alpha_k(n)|\ll_{\varepsilon} n^{\varepsilon}$
for
$(n\ge 1, \, 1\le k\le 6)
$
such that, for all $D\ge 100$ and any arithmetical function $g$, we have
\begin{equation}\label{identity:Vaughan}
\sum_{D<d\le 2D} \Lambda(d) g(d)
= S_1 + S_2 + S_3 + S_4,
\end{equation}
where 
\begin{align*}
S_1
& := \sum_{m\le D^{1/3}} \alpha_1(m) \sum_{D<mn\le 2D} g(mn),
\\\noalign{\vskip 1mm}
S_2
& := \sum_{m\le D^{1/3}} \alpha_2(m) \sum_{D<mn\le 2D} g(mn)\log n,
\\\noalign{\vskip 1mm}
S_3
& := \mathop{{\sum}\,\,{\sum}}_{\substack{D^{1/3}<m, n\le D^{2/3}\\ D<mn\le 2D}} \alpha_3(m) \alpha_4(n) g(mn),
\\\noalign{\vskip 1mm}
S_4
& := \mathop{{\sum}\,\,{\sum}}_{\substack{D^{1/3}<m, n\le D^{2/3}\\ D<mn\le 2D}} \alpha_5(m) \alpha_6(n) g(mn).
\end{align*}
The sums $S_1$ and $S_2$
are called as type I, $S_3$ and $S_4$ are called as type II.
\end{lemma}

\vskip 1mm

The third one is due to Vaaler (see \cite[Theorem A.6]{GrahamKolesnik1991}).

\begin{lemma}\label{lem:Vaaler}
Let $\psi(t) := \{t\}-\frac{1}{2}$, where $\{t\}$ means the fractional part of real number $t$.
For $x\ge 1$ and $H\ge 1$, we have
$$
\psi(x) = - \sum_{1\le |h|\le H} \Phi\Big(\frac{h}{H+1}\Big) \frac{{\rm e}(hx)}{2\pi{\rm i} h} + R_H(x),
$$
where ${\rm e}(t):={\rm e}^{2\pi{\rm i}t}$, $\Phi(t):=\pi t (1-|t|)\cot(\pi t) + |t|$ and the error term $R_H(x)$ satisfies
\begin{equation}\label{eq:lem2.2}
|R_H(x)|\le \frac{1}{2H+2} \sum_{0\le |h|\le H} \Big(1-\frac{|h|}{H+1}\Big) {\rm e}(hx).
\end{equation}
\end{lemma}

\vskip 5mm

\section{Multiple exponential sums}

Let $\alpha>0$, $\beta>0$, $\gamma>0$ and $\delta\in \R$ be some constants. 
For $X>0$, $H\ge 1$, $M\ge 1$ and $N\ge 1$, define
\begin{equation}\label{def:SdeltaHMN}
S_{\delta}
= S_{\delta}(H, M, N)
:= \sum_{h\sim H} \sum_{m\sim M} \sum_{n\sim N} a_{h, m} b_n
{\rm e}\bigg(X\frac{M^{\beta}N^{\gamma}}{H^{\alpha}} \frac{h^{\alpha}}{m^{\beta}n^{\gamma}+\delta}\bigg),
\end{equation}
where the $a_{h, m}$ and $b_n$ are complex numbers such that $|a_{h, m}|\le 1$ and $|b_n|\le 1$.
When $\delta=0$, this sum has been studied by Heath-Brown \cite{HeathBrown1983} and 
Fouvry-Iwaniec \cite{FouvryIwaniec1989}.
The aim of this section is to prove an estimate of $S_{\delta}$
by adapting and refining Heath-Brown's approach (see also \cite{LiuWu1999}).
The following proposition will play a key role in the proof of Theorem \ref{thm}.

\begin{proposition}\label{prop:HeatBrown}
Under the previous notation, for any $\varepsilon>0$ we have
\begin{align}
S_{\delta}
& \ll \big((XHMN)^{1/2} + (HM)^{1/2}N + HMN^{1/2} + X^{-1/2}HMN\big)X^{\varepsilon},
\label{BU:Sdelta_1}
\\\noalign{\vskip 1mm}
S_{\delta}
& \ll \big((X^{\kappa}H^{2+\kappa}M^{2+\kappa}N^{1+\kappa+\lambda})^{1/(2+2\kappa)}
+ HM N^{1/2}
\\
& \hskip 3,3mm
+ (HM)^{1/2} N
+ X^{-1/2}HMN\big)X^{\varepsilon} 
\label{BU:Sdelta_2}
\end{align}
uniformly for $M\ge 1$, $N\ge 1$, $H\le M^{\beta-1}N^{\gamma}$ and $|\delta|\le 1/\varepsilon$, 
where $(\kappa, \lambda)$ is an exponent pair and 
the implied constant depends on $(\alpha, \beta, \gamma, \varepsilon)$ only.
\end{proposition}

\begin{proof}
Obviously we have
$$
h^{\alpha}m^{-\beta}\le C_1H^{\alpha}M^{-\beta}=: \Xi
\qquad
(h\sim H, m\sim M),
$$
where the $C_j = C_j (\alpha, \beta, \gamma)$ denotes some positive constant depending on 
$(\alpha, \beta, \gamma)$ at most.
Let $K\ge 1$ be a parameter to be chosen later.
Introducing the following set
$$
T_k := \{(h, m) :  h\sim H, \, m\sim M, \, \Xi (k-1)<h^{\alpha}m^{-\beta}K\le \Xi k\},
$$
we can write
$$
S_{\delta}
= \sum_{n\sim N} b_n \sum_{k\le K} \sum_{(h, m)\in T_k} a_{h, m} 
{\rm e}\bigg(X\frac{M^{\beta}N^{\gamma}}{H^{\alpha}} \frac{h^{\alpha}}{m^{\beta}n^{\gamma}+\delta}\bigg).
$$
By the Cauchy-Schwarz inequality, we derive 
$$
|S_{\delta}|^2
\le NK \sum_{n\sim N} \sum_{k\le K} \sum_{(h, m)\in T_k} \sum_{(h', m')\in T_k}  a_{h, m} \overline{a_{h', m'}}\,
{\rm e}\bigg(X\frac{M^{\beta}N^{\gamma}}{H^{\alpha}} 
\Big(\frac{h^{\alpha}}{m^{\beta}n^{\gamma}+\delta}-\frac{h'^{\alpha}}{m'^{\beta}n^{\gamma}+\delta}\Big)\bigg).
$$
Inverting the order of summations, it follows that
$$
|S_{\delta}|^2
\le NK \sum_{k\le K} \sum_{(h, m)\in T_k} \sum_{(h', m')\in T_k} \big|\mathcal{S}(h, h', m, m')\big|,
$$
where 
\begin{equation}\label{def:Shm}
\mathcal{S}(h, h', m, m')
:= \sum_{n\sim N} 
{\rm e}\bigg(X\frac{M^{\beta}N^{\gamma}}{H^{\alpha}} 
\Big(\frac{h^{\alpha}}{m^{\beta}n^{\gamma}+\delta}-\frac{h'^{\alpha}}{m'^{\beta}n^{\gamma}+\delta}\Big)\bigg).
\end{equation}
Noticing that
$(h, m), (h', m')\in T_k$ imply
$\big|h^{\alpha}m^{-\beta} - h'^{\alpha}m'^{-\beta}\big|\le \Xi K^{-1}$,
we can write 
\begin{equation}\label{Sdelta2=E0+E1}
|S_{\delta}|^2
\le NK 
\mathop{\sum_{h, h'\sim H} \sum_{m, m'\sim M}}_{|h^{\alpha}m^{-\beta} - h'^{\alpha}m'^{-\beta}|\le \Xi K^{-1}}  
\big|\mathcal{S}(h, h', m, m')\big|
\le NK (S_{\delta}^{\dagger} + S_{\delta}^{\sharp}),
\end{equation}
where
\begin{align*}
S_{\delta}^{\dagger}
& := \mathop{\sum_{h, h'\sim H} \sum_{m, m'\sim M}}_{|h^{\alpha}m^{-\beta} - h'^{\alpha}m'^{-\beta}|\le 
\Xi (HM)^{-1}}  
\big|\mathcal{S}(h, h', m, m')\big|,
\\
S_{\delta}^{\sharp}
& := \mathop{\sum_{h, h'\sim H} \sum_{m, m'\sim M}}_{\Xi (HM)^{-1}<
|h^{\alpha}m^{-\beta} - h'^{\alpha}m'^{-\beta}|\le \Xi K^{-1}}  
\big|\mathcal{S}(h, h', m, m')\big|.
\end{align*}
[We have made the convention that $S_{\delta}^{\sharp}=0$ if $K\le HM$.]
Since
$$
\bigg|\frac{h^{\alpha}}{m^{\beta}} - \frac{h'^{\alpha}}{m'^{\beta}}\bigg|\le \frac{\Xi}{HM}
\;\Rightarrow\;
\bigg|\frac{h^{\alpha}}{h'^{\alpha}} - \frac{m^{\beta}}{m'^{\beta}}\bigg|
\le \frac{C_2}{\Xi}\bigg|\frac{h^{\alpha}}{m^{\beta}} - \frac{h'^{\alpha}}{m'^{\beta}}\bigg|
\le \frac{C_2}{HM},
$$
Lemma \ref{lem:space} implies that the number of $(h, h', m, m')$ verifying
$|h^{\alpha}m^{-\beta} - h'^{\alpha}m'^{-\beta}|\le \Xi (HM)^{-1}$ is
$
\ll \mathcal{D}(H, M; C_2/HM)
\ll_{\varepsilon} HM X^{\varepsilon}.
$
Thus we have trivially 
\begin{equation}\label{E0}
S_{\delta}^{\dagger}
\ll_{\varepsilon} HM N X^{\varepsilon}.
\end{equation}

Next we bound $S_{\delta}^{\sharp}$.
We write
\begin{equation}\label{function}
\begin{aligned}
& \frac{h^{\alpha}}{m^{\beta}n^{\gamma}+\delta}-\frac{h'^{\alpha}}{m'^{\beta}n^{\gamma}+\delta}
\\\noalign{\vskip 0,5mm}
& = \frac{1}{n^{\gamma}}\Big(\frac{h^{\alpha}}{m^{\beta}} - \frac{h'^{\alpha}}{m'^{\beta}}\Big)
- \frac{\delta}{n^{\gamma}}\Big(\frac{h^{\alpha}}{m^{\beta}(m^{\beta}n^{\gamma}+\delta)}
- \frac{h'^{\alpha}}{m'^{\beta}(m'^{\beta}n^{\gamma}+\delta)}
\Big)
=: f(n).
\end{aligned}
\end{equation}
Since $H\le M^{\beta-1}N^{\gamma}$, we have
$$
\bigg|\frac{h^{\alpha}}{m^{\beta}(m^{\beta}n^{\gamma}+\delta)}
- \frac{h'^{\alpha}}{m'^{\beta}(m'^{\beta}n^{\gamma}+\delta)}\bigg|
\le \frac{C_3\Xi}{M^{\beta}N^{\gamma}}
\ll \frac{\Xi}{HM}
\quad
(h, h'\sim H; \; m, m'\sim M),
$$
Therefore for $(h, h', m, m')$ verifying
$\Xi (HM)^{-1}<|h^{\alpha}m^{-\beta} - h'^{\alpha}m'^{-\beta}|\le \Xi K^{-1}$,
the first member on the right-hand side of \eqref{function} dominates the second one.
Split $(\Xi (HM)^{-1}, \Xi K^{-1}]$ 
into dyadic intervals $(\Delta\Xi, 2\Delta\Xi]$ with $1/HM\le \Delta\le 1/K$.
Take $K=\max\{2C_4X/N, 1\}$ such that 
for $(h, h', m, m')$ verifying $\Xi \Delta<|h^{\alpha}m^{-\beta} - h'^{\alpha}m'^{-\beta}|\le 2\Xi \Delta$ we have
$$
\max_{n\sim N} |f'(n)|
= C_4X\Delta N^{-1}
\le \tfrac{1}{2}\cdot
$$
By Kusmin-Landau's inequality and Lemma \ref{lem:space}, we have
\begin{equation}\label{E1:1}
S_{\delta}^{\sharp}
\ll_{\varepsilon} X^{\varepsilon} \max_{4/HM\le \Delta\le 1/K}
\Delta (HM)^2 (X\Delta N^{-1})^{-1}
\ll_{\varepsilon} X^{-1+\varepsilon} (HM)^2 N.
\end{equation}
Combining \eqref{E0} and \eqref{E1:1} with \eqref{Sdelta2=E0+E1} gives us
\begin{align*}
|S_{\delta}|^2
& \ll_{\varepsilon}
\big(HM N^2 K + X^{-1}(HMN)^2 K\big) X^{\varepsilon} 
\\\noalign{\vskip 0,5mm}
& \ll_{\varepsilon} 
\big(XHM N + HM N^2
+ (HM)^2N + X^{-1}(HMN)^2\big) X^{\varepsilon} ,
\end{align*}
which is equivalent to \eqref{BU:Sdelta_1}.

Next we prove \eqref{BU:Sdelta_2}.
If $X\le HM$, then \eqref{BU:Sdelta_1} implies that
\begin{equation}\label{BU:Sdelta_2:1}
S_{\delta}
\ll_{\varepsilon} 
\big((HM)^{1/2}N + HMN^{1/2} + X^{-1/2}HMN\big)X^{\varepsilon}.
\end{equation}
Now we can suppose that $X\ge HM$.
Applying the exponent pair $(\kappa, \lambda)$ to $\mathcal{S}(h, h', m, m')$.
and using Lemma \ref{lem:space}, we can derive that
\begin{equation}\label{E1}
\begin{aligned}
S_{\delta}^{\sharp}
& \ll_{\varepsilon} X^{\varepsilon} \max_{4/HM\le \Delta\le 1/K}
\Delta (HM)^2
\big((X\Delta N^{-1})^{\kappa} N^{\lambda} + (X\Delta N^{-1})^{-1}\big)
\\\noalign{\vskip 0,5mm}
& \ll_{\varepsilon} X^{\varepsilon} 
\big(X^{\kappa}H^2M^2N^{-\kappa+\lambda} K^{-1-\kappa} 
+ X^{-1}H^2M^2N\big).
\end{aligned}
\end{equation}
Combining \eqref{E0} and \eqref{E1} with \eqref{Sdelta2=E0+E1} gives us
$$
|S_{\delta}|^2
\ll_{\varepsilon}
\big(X^{\kappa}H^2M^2N^{1-\kappa+\lambda} K^{-\kappa}
+ HM N^2 K\big) X^{\varepsilon} 
$$
for all $K\in [1, HM]$,
where we have removed the term $X^{-1}H^2M^2N^2 K \; (\le HM N^2 K$ 
since we have suppose that $X\ge HM$).
Noticing that this estimate is trivial if $K\ge HM$,
we can optimise the parameter $K$ over $[1, \infty)$ to get
\begin{align*}
|S_{\delta}|^2
& \ll_{\varepsilon}
\big(
(X^{\kappa}H^{2+\kappa}M^{2+\kappa}N^{1+\kappa+\lambda})^{1/(1+\kappa)}
+ HM N^2\big) X^{\varepsilon}.
\end{align*}
Combining this with \eqref{BU:Sdelta_2:1}, we obtain \eqref{BU:Sdelta_2}.
\end{proof}

\vskip 5mm

\section{A key inequality}

The aim of this section is to prove the following proposition, which will play a key role for the proof of Theorem \ref{thm}.
Define
\begin{equation}\label{def:S-Lambda}
\mathfrak{S}_{\delta}(x, D)
:= \sum_{d\sim D} \Lambda(d) \psi\Big(\frac{x}{d+\delta}\Big).
\end{equation}

\begin{proposition}\label{prop}
Let $\delta\notin -\N$ be a fixed constant. Then we have
\begin{equation}\label{eq:prop_1}
\begin{aligned}
\mathfrak{S}_{\delta}(x, D)
& \ll \big((x^{2\kappa} D^{3+\lambda})^{1/(4\kappa+4)}
+ D^{5/6}
\\
& \hskip 3,4mm
+ (x^{3\kappa'} D^{-2\kappa'+2\lambda'+1})^{1/(3\kappa'+3)}
+ (x^{3\kappa'} D^{-5\kappa'+2\lambda'+1})^{1/3} 
\big) x^{\varepsilon}.
\end{aligned}
\end{equation}
uniformly for $x\ge 3$ and $1\le D\le x^{2/3}$,
where $(\kappa, \lambda)$ and $(\kappa', \lambda')$ are exponent pairs.
In particular, uniformly for $x^{6/13}\le D\le x^{2/3}$ we have
\begin{equation}\label{eq:prop_2}
\mathfrak{S}_{\delta}(x, D)
\ll_{\varepsilon} (x^2 D^7)^{1/12} x^{\varepsilon}.
\end{equation}
\end{proposition}

\begin{proof}
We apply the Vaughan identity \eqref{identity:Vaughan} with $g(d) = \psi\big(\frac{x}{d+\delta}\big)$ to write
\begin{equation}\label{proof:prop_1}
\mathfrak{S}_{\delta}(x, D)
= \mathfrak{S}_{\delta, 1} + \mathfrak{S}_{\delta, 2} + \mathfrak{S}_{\delta, 3} + \mathfrak{S}_{\delta, 4},
\end{equation}
where
\begin{align*}
\mathfrak{S}_{\delta, 1}
& := \sum_{m\le D^{1/3}} \alpha_1(m) \sum_{D<mn\le 2D} \psi\Big(\frac{x}{mn+\delta}\Big),
\\\noalign{\vskip 1mm}
\mathfrak{S}_{\delta, 2}
& := \sum_{m\le D^{1/3}} \alpha_2(m) \sum_{D<mn\le 2D} \psi\Big(\frac{x}{mn+\delta}\Big)\log n,
\\\noalign{\vskip 1mm}
\mathfrak{S}_{\delta, 3}
& := \mathop{{\sum}\,\,{\sum}}_{\substack{D^{1/3}<m, n\le D^{2/3}\\ D<mn\le 2D}}
\alpha_3(m) \alpha_4(n) \psi\Big(\frac{x}{mn+\delta}\Big),
\\\noalign{\vskip -1mm}
\mathfrak{S}_{\delta, 4}
& := \mathop{{\sum}\,\,{\sum}}_{\substack{D^{1/3}<m, n\le D^{2/3}\\ D<mn\le 2D}}
\alpha_5(m) \alpha_6(n) \psi\Big(\frac{x}{mn+\delta}\Big).
\end{align*}

Firstly we estimate $\mathfrak{S}_{\delta, 3}$.
In view of Lemma \ref{lem:Vaaler}, we can write
\begin{equation}\label{proof:prop_2}
\mathfrak{S}_{\delta, 3}
= - \frac{1}{2\pi\text{i}} \sum_{H'} \sum_{M} \sum_{N} 
\big(\mathfrak{S}^{\flat}_{\delta, 3}(H', M, N) + \overline{\mathfrak{S}^{\flat}_{\delta, 3}}(H', M, N)\big)
+ \sum_{M} \sum_{N} \mathfrak{S}^{\dagger}_{\delta, 3}(M, N).
\end{equation}
where $MN\asymp D$ (i.e. $D\ll MN\ll D$), 
$\alpha(h):=\frac{H'}{h}\Phi\big(\frac{h}{H+1}\big)\ll 1$ and
\begin{align*}
\mathfrak{S}^{\flat}_{\delta, 3}(H', M, N)
& := \frac{1}{H'} \sum_{h\sim H'} \mathop{\sum_{m\sim M} \sum_{n\sim N}}_{D<mn\le 2D} 
\alpha(h)\alpha_3(m) \alpha_4(n) \text{e}\Big(\frac{hx}{mn+\delta} \Big),
\\
\mathfrak{S}^{\dagger}_{\delta, 3}(M, N)
& := \mathop{\sum_{m\sim M} \sum_{n\sim N}}_{D<mn\le 2D} \alpha_3(m) \alpha_4(n) R_H\Big(\frac{x}{mn+\delta}\Big).
\end{align*}

Firstly we bound $\mathfrak{S}^{\flat}_{\delta, 3}(H', M, N)$.
We remove the extra multiplicative condition $D<mn\le 2D$ at the cost of a factor $\log x$.
On the other hand, 
in view of the symmetry of the variables $m$ and $n$, we can suppose that 
\begin{equation}\label{Cond:MND}
D^{1/3}\le M\le D^{1/2}\le N\le D^{2/3}.
\end{equation}
By Proposition \ref{prop:HeatBrown} with $\alpha=\beta=\gamma=1$ and $(X, H, M, N)=(xH'/MN, H', M, N)$, 
we have
\begin{equation}\label{S3:1}
\begin{aligned}
\mathfrak{S}^{\flat}_{\delta, 3}(H', M, N)
& \ll_{\varepsilon}
\big((x^{\kappa} M^2N^{1+\lambda})^{1/(2+2\kappa)}
+ M N^{1/2}
+ M^{1/2} N 
+ (x^{-1}DH')^{1/2}\big) x^{\varepsilon},
\end{aligned}
\end{equation}
provided $H'\le H\le N$.

Secondly we bound $\mathfrak{S}^{\dagger}_{\delta, 3}(M, N)$.
Using \eqref{eq:lem2.2}, we have
\begin{align*}
\mathfrak{S}^{\dagger}_{\delta, 3}(M, N)
& \ll x^{\varepsilon} \sum_{m\sim M} \sum_{n\sim N}
\Big|R_H\Big(\frac{x}{mn+\delta}\Big)\Big|
\\
& \ll \frac{x^{\varepsilon}}{H} \sum_{m\sim M} \sum_{n\sim N}
\sum_{0\le |h|\le H} \Big(1-\frac{|h|}{H+1}\Big) {\rm e}\Big(\frac{hx}{mn+\delta}\Big)
\\
& \ll x^{\varepsilon}\big(DH^{-1} + \max_{1\le H'\le H} \big|\widetilde{\mathfrak{S}}^{\dagger}_{\delta, 3}(H', M, N)\big|\big),
\end{align*}
where
$$
\widetilde{\mathfrak{S}}^{\dagger}_{\delta, 3}(H', M, N)
:= \frac{1}{H} \sum_{m\sim M} \sum_{n\sim N}
\sum_{h\sim H'} \Big(1-\frac{|h|}{H+1}\Big) {\rm e}\Big(\frac{hx}{mn+\delta}\Big).
$$
Clearly we can bound $\widetilde{\mathfrak{S}}^{\dagger}_{\delta, 3}(H', M, N)$
in the same way as $\mathfrak{S}^{\flat}_{\delta, 3}(H', M, N)$
and obtain
\begin{equation}\label{S3:2}
\begin{aligned}
\mathfrak{S}^{\dagger}_{\delta, 3}(M, N)
& \ll_{\varepsilon}
\big(DH^{-1} 
+ (x^{\kappa} M^2N^{1+\lambda})^{1/(2+2\kappa)}
\\
& \hskip 4,5mm
+ M N^{1/2}
+ M^{1/2} N 
+ (x^{-1}DH)^{1/2}\big) x^{\varepsilon},
\end{aligned}
\end{equation}
provided $H\le N$.
Combining \eqref{S3:1} and \eqref{S3:2} with \eqref{proof:prop_2} and using \eqref{Cond:MND}, we find that
\begin{align*}
\mathfrak{S}_{\delta, 3}
& \ll_{\varepsilon}
\big(DH^{-1} 
+ (x^{\kappa} M^2N^{1+\lambda})^{1/(2+2\kappa)}
+ M N^{1/2}
+ M^{1/2} N
+ (x^{-1}DH)^{1/2}\big) x^{\varepsilon}
\\
& \ll_{\varepsilon}
\big(DH^{-1} 
+ (x^{2\kappa} D^{3+\lambda})^{1/(4+4\kappa)}
+ D^{5/6}
+ (x^{-1}DH)^{1/2}\big) x^{\varepsilon},
\end{align*}
provided $H\le D^{1/2} \, (\le N)$.
Optimising $H$ over $[1, D^{1/2}]$, it follows that
\begin{equation}\label{proof:prop_10}
\mathfrak{S}_{\delta, 3}
\ll_{\varepsilon}
\big((x^{2\kappa} D^{3+\lambda})^{1/(4+4\kappa)}
+ D^{5/6}
+ (x^{-1}D^2)^{1/3}\big) x^{\varepsilon}.
\end{equation}
Clearly the same estimate also holds for $\mathfrak{S}_{\delta, 4}$.

On the other hand, let $(\kappa', \lambda')$ be an exponent pair,
then \cite[(5.16)]{MaWu2020} gives us
\begin{equation}\label{proof:prop_14}
\mathfrak{S}_{\delta, j}
\ll \big((x^{3\kappa'} D^{-2\kappa'+2\lambda'+1})^{1/(3\kappa'+3)} + x^{\kappa'} D^{(-5\kappa'+2\lambda'+1)/3} 
+ x^{-1} D^2\big) x^{\varepsilon}.
\end{equation}
for $j=1, 2$.
Inserting \eqref{proof:prop_14} and \eqref{proof:prop_10} into \eqref{proof:prop_1}
and using the fact that 
$$
\max\{(x^{-1}D^2)^{1/3}, x^{-1} D^2\}\le D^{5/6}
$$ 
for $1\le D\le x^{2/3}$,
we get \eqref{eq:prop_1}.

Taking $(\kappa, \lambda)=(\kappa', \lambda')=(\frac{1}{2}, \frac{1}{2})$ in \eqref{eq:prop_1},
we find that 
$$
\mathfrak{S}_{\delta}(x, D)
\ll \big((x^2 D^7)^{1/12}
+ D^{5/6}+ (x^3 D^2)^{1/9} + (x^3 D^{-1})^{1/6}
\big) x^{\varepsilon},
$$
which implies \eqref{eq:prop_2}, 
since the last three terms can be absorbed by the first one provided $x^{6/13}\le D\le x^{2/3}$.
\end{proof}

\vskip 5mm

\section{Proof of Theorem \ref{thm}}

Let $N\in [x^{6/13}, x^{1/2})$ be a parameter which can be chosen later.
First we write
\begin{equation}\label{3.1}
\sum_{n\le x} \Lambda\Big(\Big[\frac{x}{n}\Big]\Big) := S_1(x)+S_2(x)
\end{equation}
with
$$
S_1(x):=\sum_{n\le N} \Lambda\Big(\Big[\frac{x}{n}\Big]\Big),      
\qquad 
S_2(x):=\sum_{N<n\le x} \Lambda\Big(\Big[\frac{x}{n}\Big]\Big).
$$
We have trivially
\begin{equation}\label{3.2}
S_1(x)
\ll_{\varepsilon} N x^{\varepsilon}.
\end{equation}
Next we bound $S_2(x)$. Putting $d=[x/n]$, then
$x/n-1<d\le x/n
\Leftrightarrow
x/(d+1)<n\le x/d$.
Thus we can write
\begin{equation}\label{2.4}
\begin{aligned}
S_2(x)
& =\sum_{d\le x/N} \Lambda(d) \sum_{x/(d+1)<n\le x/d} 1
\\
& =\sum_{d\le x/N} \Lambda(d) \Big(\frac{x}{d}-\psi\Big(\frac{x}{d}\Big)-\frac{x}{d+1}+\psi\Big(\frac{x}{d+1} \Big) \Big)
\\
& = x\sum_{d\ge 1} \frac{\Lambda(d)}{d(d+1)} +    \mathcal{R}_{1}(x) -  \mathcal{R}_{0}(x) + O(N),
\end{aligned}
\end{equation}
where we have used the following bounds
$$
x \sum_{d>x/N} \frac{\Lambda(d)}{d(d+1)}\ll_{\varepsilon} N x^{\varepsilon},
\qquad
\sum_{d\le N} \Lambda(d)\Big(\psi\Big(\frac{x}{d+1}\Big) - \psi\Big(\frac{x}{d}\Big)\Big)
\ll_{\varepsilon} N x^{\varepsilon}
$$
and defined
$$
\mathcal{R}_{\delta}(x)
= \sum_{N< d\le x/N} \Lambda(d) \psi\Big(\frac{x}{d+\delta}\Big).
$$
Writing $D_j:=x/(2^jN)$, we have $x^{6/13}\le N\le D_j\le x/N\le x^{7/13}$ for $0\le j\le \frac{\log(x/N^2)}{\log 2}$
since $x^{6/13}\le N\le x^{1/2}$.
Thus we can apply \eqref{eq:prop_2} of Proposition \ref{prop} to get
\begin{align*}
|\mathcal{R}_{\delta}(x)|
& \le \sum_{0\le j\le \log(x/N^2)/\log 2} |\mathfrak{S}_{\delta}(x, D_j)|
\\
& \ll \sum_{0\le j\le \log(x/N^2)/\log 2} (x^2D_j^7)^{1/12} x^{\varepsilon}
\\
& \ll (x^9N^{-7})^{1/12} x^{\varepsilon}.
\end{align*}
Putting this into \eqref{2.4} and taking $N=x^{9/19}$, we find that
\begin{equation}\label{S2}
S_2(x) = x\sum_{d\ge 1} \frac{\Lambda(d)}{d(d+1)} + O_{\varepsilon}\big(x^{9/19+\varepsilon}\big).
\end{equation}
Inserting \eqref{3.2} with $N=x^{9/19}$ and \eqref{S2} into \eqref{3.1}, we get the required result.

\vskip 3mm

\noindent{\bf Acknowledgement}.
This work is supported in part by the National Natural Science Foundation of China 
(Grant Nos. 12071238, 11771252, 11971370 and 12071375)
and by the NSF of Chongqing (Grant  No. cstc2019jcyj-msxm1651).

\end{document}